\documentclass[12 pt,reqno]{amsart}

\usepackage{mathrsfs,amsfonts,amssymb,amsmath, xcolor}
\usepackage{graphicx,cite}

\newcommand{\R}{\mathbb{R}}
\newcommand{\N}{\mathbb{N}}

\newcommand{\eps}{ \varepsilon}

\newcommand{\supp}{\textrm{supp }}

\numberwithin{equation}{section}

\newtheorem{theorem}{Theorem}[section]
\newtheorem*{theorem*}{Theorem}
\newtheorem{corollary}{Corollary}[section]
\newtheorem{proposition}{Proposition}[section]
\newtheorem{lemma}{Lemma}[section]

\theoremstyle{definition}
\newtheorem{definition}{Definition}[section]

\author{Wentao Cao}
\address{Institut f\"{u}r mathematik, Universit\"{a}t Leipzig, D-04109, Leipzig, Germany}
\email{wentao.cao@math.uni-leipzig.de}
\author{L\'aszl\'o Sz\'ekelyhidi Jr.}
\address{Institut f\"{u}r mathematik, Universit\"{a}t Leipzig, D-04109, Leipzig, Germany}
\email{laszlo.szekelyhidi@math.uni-leipzig.de}

\title[$C^{1,\alpha}$ isometric extensions]
{$C^{1,\alpha}$ isometric extensions}

\date{\today}


\begin{document}

\begin{abstract}
In this paper we consider the Cauchy problem for isometric immersions. More precisely, given a smooth isometric immersion of a codimension one submanifold we construct $C^{1,\alpha}$ isometric extensions for any $\alpha<\frac{1}{n(n+1)+1}$ via the method of convex integration.
\end{abstract}

\maketitle

%
\section{Introduction}
The problem of extending isometric embeddings is considered in this paper. More precisely, let $\Sigma$ be a codimension one submanifold of a smooth $n$-dimensional Riemannian manifold $(\mathcal{M}, g)$
and let $f:\Sigma\rightarrow\R^q$ be a smooth isometric immersion, $q\geq n+1.$ The problem is to extend $f$ to be an isometric immersion $u:\mathcal{M}\rightarrow\R^q.$ This problem was first considered by Jacobowitz in \cite{Jac74} when the target dimension is $q=n_*=\frac{1}{2}n(n+1)$ in the high-regularity setting. In particular, Jacobowitz showed that this extension problem can viewed as a Cauchy problem, and showed that local extensions can be obtained under certain geometric conditions (i) by a variant of the Cauchy-Kowalevskaja Theorem in the analytic setting and (ii) by an adaptation of Nash's implicit function theorem in the $C^k$ setting with $k\geq17.$ The geometric sufficient condition of Jacobowitz (see below for precise statement) can be seen as stating that the image $f(\Sigma)$ has to be ``more curved" than $\Sigma$, and it was shown in \cite{Jac74} that this condition is almost necessary -- indeed, if $f(\Sigma)$ is geodesic but $\Sigma$ is not, no isometric extension can exist. 

Viewing the extension problem as an initial-value problem has also played an important role in progress concerning Schl\"afli's conjecture (see \cite{Gr17}) on the local embedding of surfaces in $\R^3$. 
For extending from one point where the Gaussian curvature is changing signs in a 2-dimensional manifold, sufficiently smooth embeddings into $\R^3$ have been constructed in a small neighbourhood of such point in \cite{Lin, Han05} by imposing some non-degeneracy conditions on the vanishing of the curvature. Similarly, for the isometric extension from a small curve across which  Gaussian curvature vanishes at some order, sufficiently smooth local or semi-global isometric embeddings  into $\R^3$ also exist in some neighbourhood of such curve under some conditions, see \cite{Han06, Kh, Dong}. 

The isometric extension problem can also be considered in the low-regularity and low codimension setting. For the classical global isometric immersion problem of $u:\mathcal{M}\rightarrow\R^{n+1}$ the celebrated Nash-Kuiper Theorem  \cite{Nash54, Kui55} provides solutions of class $u\in C^1$. More precisely, any smooth strictly short immersion $\bar u:\mathcal{M}\rightarrow\R^{n+1}$ may be uniformly approximated by $C^1$ isometric immersions -- this high flexibility is a paradigm example of Gromov's h-principle \cite{Gro86}. We recall that being isometric amounts to, in local coordinates, to the following system of first order partial differential equations: 
$$\frac{\partial u}{\partial x_i}\cdot\frac{\partial u}{\partial x_j}=g_{ij},$$
where $g=\sum_{i, j=1}^ng_{ij}dx^idx^j$ is  the metric in local coordinates, whereas $u$ is strictly short if the metric tensor $g_{ij}-\frac{\partial u}{\partial x_i}\cdot\frac{\partial u}{\partial x_j}$ is a positive definite tensor. The Nash-Kuiper Theorem has been extended to  $C^{1, \alpha}$  in 
\cite{Bor65} and subsequently in \cite{CDS12} for $\alpha<(2n_*+1)^{-1}$ in the local case and for $\alpha<(2(n+2)n_*+1)^{-1}$ in the global case, where $n_*=\frac{1}{2}n(n+1)$. Further progress in the local 2-dimensional case was achieved in \cite{DIS}, for $\alpha<1/5$. Let us mention that the interest in $C^{1, \alpha}$ immersions is two-fold. On the one hand  there is an interesting dichotomy between the high flexibility provided by the Nash-Kuiper Theorm for $C^{1, \alpha}$ with small $\alpha$ and extensions of classical rigidity theorems \cite{Bor04, CDS12} for $\alpha$ close to 1. On the other hand there is an unexpected parallel between the theory $C^{1, \alpha}$ isometric immersions and the theory of $C^{\alpha}$ weak solutions of the incompressible Euler equations and more general classes of hydrodynamic equations-this connection eventually leads to the resolution of Onsager's conjecture in 3-dimensional turbulence, see \cite{B16, BDIS15,BDLSV17, Ise16, DS14}. The Cauchy problem for the Euler equations has been considered in \cite{D14, DS17}.

In this paper, we aim to explore more in this parallel direction, that is, study the local extension problem in low codimension and low regularity. In \cite{HunWas16} Hungerb\"uhler and Wasem first considered this problem and showed that in the $C^1$ category the analogue of the Nash-Kuiper statement is valid for one-sided isometric extensions. In this paper we consider $C^{1, \alpha}$ isometric extensions from adapted short immersions (see definition below). In particular, we also construct a one-sided  adapted short immersions from short immersion, and  then a $C^{1, \alpha}$ isometric extension. 

\subsection{Main results and ideas.}
In order to keep the presentation simple, we focus on the case of a single chart. That is, $\Omega\subset\R^n$ is an open bounded set equipped with a smooth metric $g=(g_{ij})$. In the coordinates $x_1,\dots,x_n\in\Omega$ we denote the derivative of a map $u:\Omega\to\R^{n+1}$ by $\nabla u=(\partial_ju_i)_{i,j}$. Then, a differentiable map $u$ is isometric if $g=\nabla u^T\nabla u$ and strictly short 
if $g-\nabla u^T\nabla u>0$, i.e. the $n\times n$ matrix $(g_{ij}-\partial_iu\cdot \partial_ju)_{ij}$ is positive definite for all $x$. As usual, we call $u$ an immersion if $\nabla u(x)$ has rank $n$ for every $x$, and an embedding if in addition $u$ is 1-1.

In analogy with \emph{adapted subsolutions} introduced in \cite{DS17}, we define \emph{adapted short immersions} as follows.
\begin{definition}\label{d:adapted}
Let $\Omega\subset\R^n$ be a bounded open set equipped with a smooth metric $g=(g_{ij})$. We call a map $u:\overline{\Omega}\rightarrow\R^{n+1}$ a $C^{1, \alpha}$-adapted short immersion (embedding) with constants $M,r>0$, if $u\in C^{1, \alpha}(\overline{\Omega})$ is a short immersion (embedding) in $\overline{\Omega}$ with
$$
g-\nabla u^T\nabla u=\rho^2(Id+G)
$$
with a nonnegative function $\rho\in C(\overline{\Omega})$ and symmetric tensor $G\in C(\overline{\Omega};\R^{n\times n})$ satisfying 
\begin{equation}\label{e:g-adapt0}
|G(x)|\leq r\quad\textrm{ for all }x\in\overline{\Omega}
\end{equation}
and the following additional estimates hold: $u\in C^2(\overline{\Omega}\setminus\{\rho=0\})$ and $\rho,G\in C^1(\overline{\Omega}\setminus\{\rho=0\})$ with 
\begin{align}
|\nabla^2u(x)|&\leq M\rho(x)^{-2}, \label{e:u-adapt}\\
|\nabla\rho(x)|&\leq M\rho(x)^{-1},\label{e:rho-adapt} \\
\quad |\nabla G(x)|&\leq M\rho(x)^{-3},\label{e:g-adapt}
\end{align}
for any $x\in\overline{\Omega}\setminus \{\rho=0\}$. 
\end{definition}

One of our main results is to show how the existence of a $C^{1,\alpha}$ extension can be reduced to the existence of an $C^{1, \alpha}$ adapted short immersion.

\begin{theorem}{\bf [Isometric approximation]}\label{t:isometricextend}
Let $\Omega\subset\R^n$ be a bounded open set equipped with a smooth metric $g=(g_{ij})$ and let $\Sigma\subset\overline{\Omega}$ be a $C^2$ codimension one submanifold. Let $u:\overline{\Omega}\to\R^{n+1}$ be a $C^{1, \alpha_0}$ adapted short immersion for some $\alpha_0\geq\frac{1}{n(n+1)+1}$ with $\nabla u^T\nabla u=g$ on $\Sigma$. Then, for any $\epsilon>0$ and any $\alpha<\frac{1}{n(n+1)+1}$, there exists a $C^{1, \alpha}$ isometric immersion $\bar{u}:\overline{\Omega}\to\R^{n+1}$ such that  $\|\bar{u}-u\|_{C^0(\Omega)}<\epsilon$.
\end{theorem}

The basic method for obtaining $C^{1,\alpha}$ isometric immersions has been developed in \cite{CDS12}. Indeed, the bound $\alpha<\frac{1}{n(n+1)+1}=\frac{1}{2n_*+1}$ on the H\"older exponent agrees with the bound obtained in \cite{CDS12}. However, we need to substantially modify the proof in order to make sure $\bar{u}=u$ on $\Sigma$. This requires us to modify the mollification argument (see Section \ref{s:iteration}), which is designed to handle the problem with loss of derivatives. Then we adapt the strategy introduced in \cite{DS17} of decomposing the domain $\overline{\Omega}\setminus\Sigma$ in level sets of the metric error (see Section \ref{s:extension}). 
  
In turn, one needs to consider criteria under which such an adapted short extension exists. In this paper we give a sufficient criterion, following the approach in \cite{HunWas16}. To recall the general setting, let $(\mathcal{M},g)$ be a smooth Riemannian manifold with a smooth codimension 1 submanifold $\Sigma$ and an isomeric immersion $f:\Sigma\to\R^{n+1}$. For any point $p\in \Sigma$  
 we denote, as usual, by $T_p\Sigma$ the tangent vector space of $\Sigma\subset\mathcal{M}$, and $L: T_p\Sigma\times T_p\Sigma\rightarrow\R^1$ and $\bar{L}: T_p\Sigma\times T_p\Sigma\rightarrow N_{f(p)}$ are the second fundamental forms of $\Sigma\subset\mathcal{M}$ and $f(\Sigma)\subset\R^{n+1}$ respectively, with $N_{f(p)}$ denoting the normal space to $f(\Sigma)$ at $f(p).$ Here $\langle\cdot, \cdot\rangle$ denotes Euclidean scalar product.

In \cite{Jac74} Jacobowitz showed that a sufficient condition for the existence of a smooth isometric extension of $f:\Sigma\to\R^{n+1}$ is that
there exists a vectorfield $\mu:\Sigma\to\R^{n+1}$ with
\begin{equation}\label{e:J}
\begin{split}
\textrm{(i)}&\quad \mu(p)\in N_{f(p)},\\
\textrm{(ii)}&\quad |\mu(p)|<1,\\
\textrm{(iii)}&\quad \langle\mu(p), \bar{L}(\cdot, \cdot)\rangle =L(\cdot,\cdot)\quad\textrm{Êon }T_p\Sigma
\end{split}
\end{equation}
for any $p\in\Sigma$. Moreover, a simple example in \cite{Jac74} shows that condition \eqref{e:J} is optimal in the sense that in general one cannot replace (ii) with $|\mu(p)|\leq 1$. On the other hand, in \cite{HunWas16} Hungerb\"uhler-Wasem showed that a sufficient condition for the existence of a $C^1$ \emph{one-sided} isometric extension is that
there exists a vectorfield $\mu:\Sigma\to\R^{n+1}$ with
\begin{equation}\label{e:HW}
\begin{split}
\textrm{(i)}&\quad \mu(p)\in N_{f(p)},\\
\textrm{(ii)}&\quad |\mu(p)|=1,\\
\textrm{(iii)}&\quad \langle\mu(p), \bar{L}(\cdot, \cdot)\rangle -L(\cdot,\cdot)\textrm{ is positive definite on }T_p\Sigma
\end{split}
\end{equation}
for any $p\in\Sigma$. We recall that a one-sided extension is defined in \cite{HunWas16} as the extension of $f$ to a one-sided neighbourhood of $\Sigma$. The latter can be defined as follows.
Let $B$ be a ball centered at zero in $\R^n$, which is a local chart of $\mathcal{M}$ for an open set containing $\Sigma$ and let $B_0=\bar{B}\cap(\R^{n-1}\times\{0\})$, then  $\bar{B}\cap(\R^{n-1}\times\R_{\leq0})$ and $\bar{B}\cap(\R^{n-1}\times\R_{\geq0})$ are both called one-sided neighbourhoods of $B_0$.  A one-sided neighbourhood $\Omega$ of a point in $\Sigma$ is the image of a one-sided neighbourhood of $B_0$ under the inverse of such local chart.

Our second main result is the existence of an adapted short extension under the condition \eqref{e:HW}.

\begin{theorem} {\bf [Adapted short extension]}\label{t:shortextend}
Let $f:\Sigma\rightarrow\R^{n+1}$ be a smooth isometric immersion, where $\Sigma$ is a codimension one submanifold of an $n$-dimensional Riemannian manifold $(\mathcal{M}, g)$. If there exists a vectorfield $\mu$ satisfying condition \eqref{e:HW} at any point $p\in\Sigma$, then there exists a one-sided neighbourhood $\Omega$ of $\Sigma$ at $p$, such that for any $\alpha_0<1/3$ there exists a $C^{1, \alpha_0}$ adapted short immersion $u:(\overline{\Omega},g)\to\R^{n+1}$ with $\nabla u^T\nabla u=g$ and $u=f$ on $\Sigma$.
\end{theorem}

Obviously, combining Theorem \ref{t:shortextend} with Theorem \ref{t:isometricextend}, one can easily obtain
\begin{theorem}{\bf [Isometric extension]}
Assume $f:\Sigma\rightarrow\R^{n+1}$ is a smooth isometric immersion, where $\Sigma$ is a codimension one submanifold of an $n$-dimensional Riemannian manifold $(\mathcal{M}, g)$. If there exists a vectorfield $\mu$ satisfying condition \eqref{e:HW} at any point $p\in\Sigma$, then there exists a one-sided neighbourhood $\Omega$ of $\Sigma$ at $p$, such that for any $\alpha<\frac{1}{n(n+1)+1}$ there exists a $C^{1, \alpha}$ isometric immersion $u:(\overline{\Omega},g)\to\R^{n+1}$ with $u=f$ on $\Sigma$.
\end{theorem}

We expect that in the case $n=2$ the techniques of \cite{DIS} can be adapted to construct $C^{1, \alpha}$ isometric extensions for any $\alpha<1/5$ (rather than just $\alpha<1/7$). 

\section{Preliminaries}\label{s:preliminary}
In this section we introduce some notation, function spaces and basic lemmas. For a function $f:\Omega\to\R^q$ and for a multi-index $\beta$ the H\"older norms are defined as follows:
\begin{equation*}
\|f\|_0=\sup_{\Omega}f, ~~\|f\|_m=\sum_{j=0}^m\max_{|\beta|=j}\|\partial^\beta f\|_0,
\end{equation*}
and
\begin{equation*}
[f]_{\alpha}=\sup_{x\neq y}\frac{|f(x)-f(y)|}{|x-y|^\alpha},~~[f]_{m+\alpha}=\max_{|\beta|=m}\sup_{x\neq y}\frac{ |\partial^\beta f(x)-\partial^\beta f(y)|}{|x-y|^\alpha}, 0<\alpha\leq1.
\end{equation*}
Then the H\"older norms are given as
$$\|f\|_{m+\alpha}=\|f\|_m+[f]_{m+\alpha}. $$
We recall the standard interpolation inequality
$$
[f]_r\leq C\|f\|_0^{1-\frac{r}{s}}[f]_s^{\frac{r}{s}}
$$
for $s>r\geq0$ and the approximation of H\"older functions by smooth functions:
\begin{lemma}\label{l:mollification}
For any $r, s\geq0,$ and $0<\alpha\leq1,$ we have
\begin{align*}
&[f*\varphi_l]_{r+s}\leq Cl^{-s}[f]_r,\\
&\|f-f*\varphi_l\|_r\leq Cl^{1-r}[f]_1, \text{ if } 0\leq r\leq1,\\
&\|(fg)*\varphi_l-(f*\varphi_l)(g*\varphi_l)\|_r\leq Cl^{2\alpha-r}\|f\|_\alpha\|g\|_\alpha,
\end{align*}
with constant $C$ depending only on $s, r, \alpha, \varphi.$
\end{lemma}
In the above $\varphi_l$ is a standard mollifying kernel at length-scale $l>0$.
Other properties about H\"older norm can be found in references such as \cite{CDS12,DS14,DIS}. 

For an $n\times n$ matrix $P$ we will use the operator norm as defined by
$$
|P|:=\sup_{\xi\in S^{n-1}}|P\xi|.
$$
The $n\times n$ identity matrix is denoted by $\textrm{Id}$.
We recall the following lemmas. The first one is the decomposition of a metric into a sum of primitive metrics \cite{Nash54}, where we refer to the form used in \cite{CDS12, DIS}.
 
\begin{lemma}\label{l:deco}
For any $n\geq 2$ there exists a geometric constant $r_0>0$, vectors $\nu_1,\dots,\nu_{n_*}\in S^{n-1}$ and smooth functions $a_k\in C^\infty(B_{r_0}(Id))$ such that, for any positive definite matrix $P\in \R^{n\times n}$ with
\begin{equation}\label{e:deco}
|P-Id|\leq r_0,
\end{equation}
the identity
$$
P=\sum_{k=1}^{n_*}a_k^2(P)\nu_k\otimes\nu_k
$$
holds.
\end{lemma}
Here we recall that $n_*=\frac{1}{2}n(n+1)$.
Of course we may assume without loss of generality that $r_0<1$.
For future reference we fix some radii $0<r_1<r_2<r_0$ such that
\begin{equation}\label{e:radius}
r_1\leq\frac{r_2}{5}\leq\frac{1}{25}r_0,
\end{equation}
with $r_0$ from Lemma \ref{l:deco}.

Finally, we recall the following lemma, which describes the profile of the building block of convex integration (the corrugation) for the isometric immersion problem, from \cite{CDS12}, see also \cite{HunWas16}.
\begin{lemma}\label{l:Gamma}
There exist some $\delta_*>0$ and a pair of functions $(\Gamma_1, \Gamma_2)\in C^{\infty}([0, \delta_*]\times\R, \R^2)$ such that $(\Gamma_1, \Gamma_2)(s, t)=(\Gamma_1, \Gamma_2)(s, t+2\pi)$ and
\begin{equation*}
\begin{split}
&(1+\partial_t\Gamma_1)^2+(\partial_t\Gamma_2)^2=1+s^2;~~\\
&\|\partial_s\partial_t^k\Gamma_1(s, t)\|_0+\|\partial_t^k(\Gamma_1, \Gamma_2)(s, t)\|_0\leq C(k)s, \text{ for } k\geq0.
\end{split}
\end{equation*}
\end{lemma}

\section{Main Iteration Propositions}\label{s:iteration}
In this section, we state and prove our main iteration propositions which form the building block of the convex integration iteration. We start by recalling how to ``add'' a primitive metric - this is called a ``step'' in the terminology of \cite{Nash54}, see also Proposition 2 in \cite{CDS12}.


\begin{proposition}\label{p:step} {\bf[Step]}
Let $u\in C^{2}(\Omega, \R^{n+1})$ be an immersion, $\nu\in S^{n-1}$ and $a\in C^{2}(\Omega)$. Assume that
\begin{align}
\frac{1}{\gamma}\textrm{Id}\leq\nabla u^T&\nabla u\leq\gamma \textrm{Id}\quad\textrm{ in }\Omega\label{e:step-1a}\\
\quad \|u\|_{2}&\leq M\delta^{1/2}\theta,\label{e:step-2a}\\
\|a\|_0\leq(\tfrac{1}{2}\gamma\eps)^{1/2},\quad \|a\|_1&\leq M\eps^{1/2}\theta, \quad \|a\|_2\leq M\eps^{1/2}\theta\tilde\theta,\label{e:step-3a}
\end{align}
for some $M,\gamma\geq 1$, $\eps\leq\delta\leq 1$ and $\theta\leq\tilde\theta$.
There exists a constant $c_0=c_0(M,\gamma)$ such that, for any 
\begin{equation}\label{e:constraint-on-lambda}
\lambda\geq c_0\frac{\delta^{1/2}}{\eps^{1/2}}\tilde\theta,
\end{equation}
there exists an immersion $v\in C^{2}(\Omega, \R^{n+1})$ such that 
\begin{align}
\frac{1}{2\gamma}\textrm{Id}\leq\nabla v^T&\nabla v\leq2\gamma \textrm{Id}\quad\textrm{ in }\Omega\label{e:step-0}\\
v&=u\textrm{ on }\Omega\setminus \textrm{supp }a,\label{e:step-1}\\
\|v-u\|_j&\leq \overline{M}\eps^{1/2}\lambda^{j-1}, j=0, 1, 2\label{e:step-2}\\
\|v\|_2&\leq \overline{M}\eps^{1/2}\lambda,\label{e:step-3}\\
\|\nabla v^T\nabla v-(\nabla u^T\nabla u+a^2\nu\otimes\nu)\|_j&\leq C(M,\gamma)\eps^{1/2}\delta^{1/2}\theta\lambda^{j-1}, j=0, 1.\label{e:step-4}
\end{align}
Here $\overline{M}$ is a constant depending only on $\gamma$.
\end{proposition}


\begin{proof}
Let us fix $\lambda\geq c_0\frac{\delta^{1/2}}{\eps^{1/2}}\tilde\theta\geq\tilde\theta$, where $c_0$ will be chosen later during the proof. For the moment it suffices to assume that $c_0\geq 1$, so that, in particular, in the sequel we may assume the inequalities
\begin{equation}\label{e:step-parameter-inequalities}
\theta\leq  \frac{\delta^{1/2}}{\eps^{1/2}}\theta\leq \lambda.
\end{equation}
We start by regularizing $u$ on length-scale $\lambda^{-1}$ to get a smooth immersion $\tilde{u}$ satisfying
\begin{equation}\label{e:tildeu-C2}
\|\tilde{u}-u\|_1\leq C(M)\delta^{1/2}\theta\lambda^{-1} ,\quad \|\tilde{u}\|_{2}\leq C(M)\delta^{1/2}\theta,\quad \|\tilde{u}\|_{3}\leq C(M)\delta^{1/2}\theta\lambda.
\end{equation}
Observe that 
$$
\nabla\tilde u^T\nabla\tilde u=\nabla u^T\nabla u-(\nabla u-\nabla\tilde u)^T\nabla u-\nabla\tilde u^T(\nabla u-\nabla\tilde u),
$$
and hence
\begin{equation}\label{e:tildeu-C1}
\frac{1}{2\gamma}\textrm{Id}\leq \nabla\tilde u^T\nabla \tilde u\leq 2\gamma \textrm{Id},
\end{equation}
provided $\delta^{1/2}\theta\lambda^{-1}\leq c_2^{-1}$ for some $c_2=c_2(M,\gamma)$. Choosing $c_0\geq c_2$ in the inequality \eqref{e:constraint-on-lambda} constraining $\lambda$ will ensure this. Then it follows that $\nabla\tilde u^T\nabla \tilde u$ is invertible, and hence we can set
\begin{align*}
&\tilde{\xi}=\nabla \tilde{u}(\nabla \tilde{u}^T\nabla \tilde{u})^{-1}\nu,\quad \xi=\frac{\tilde{\xi}}{|\tilde{\xi}|^2},\\
&\tilde{\zeta}=*(\partial_1\tilde{u}\wedge\partial_2\tilde{u}\wedge\cdots\wedge\partial_n\tilde{u}),\quad 
\zeta=\frac{\tilde{\zeta}}{|\tilde{\xi}\|\tilde{\zeta}|},\\
&\tilde{a}=|\tilde{\xi}|a,
\end{align*}
where $*$ denotes the Hodge star operator. Observe that we then have, by construction, 
\begin{equation}\label{e:xizeta}
\nabla \tilde u^T\xi=\frac{1}{|\tilde\xi|^2}\nu,\quad \nabla\tilde u^T\zeta=0.
\end{equation}
It follows from \eqref{e:tildeu-C2}-\eqref{e:tildeu-C1} that 
\begin{equation}\label{e:xizeta-est}
\begin{split}
\|(\xi, \zeta)\|_0&\leq C(\gamma), \\
\|(\xi, \zeta)\|_1&\leq C(\gamma, M)\delta^{1/2}\theta, \\
\|(\xi, \zeta)\|_2&\leq C(\gamma, M)\delta^{1/2}\theta\lambda, 
\end{split}
\end{equation}
and
\begin{equation}\label{e:acm}
\begin{split}
\|\tilde{a}\|_0&\leq C(\gamma)\eps^{1/2},\\
\|\tilde{a}\|_1&\leq C(\gamma)(\|a\|_1\|~|\tilde{\xi}|~\|_0+\|a\|_0\|~|\tilde{\xi}|~\|_1)\\
&\leq C(\gamma, M)(\eps^{1/2}\theta+\eps^{1/2}\delta^{1/2}\theta)\\
&\leq C(\gamma, M)\eps^{1/2}\theta,\\
\|\tilde{a}\|_2&\leq C(\gamma)(\|a\|_2\|~|\tilde{\xi}|~\|_0+\|a\|_0\|~|\tilde{\xi}|~\|_2)\\
&\leq C(\gamma, M)(\eps^{1/2}\theta\tilde\theta+\eps^{1/2}\delta^{1/2}\theta\lambda)\\
&\leq C(\gamma, M)\eps^{1/2}\theta\lambda,
\end{split}
\end{equation}
where in the final inequality we have used that 
$\tilde\theta\leq \lambda$.
We set
$$
v=u+\frac{1}{\lambda}\bigl(\Gamma_1(\tilde{a}, \lambda x\cdot\nu)\xi+\Gamma_2(\tilde{a}, \lambda x\cdot\nu)\zeta\bigr).
$$
Using Lemma \ref{l:Gamma} we see that $v=u$ outside $\textrm{supp }a$, so that \eqref{e:step-1} holds.

 Next, note that
\begin{equation}\label{e:interu}
\begin{split}
\|v-u\|_{j}&\leq\frac{1}{\lambda}(\|\Gamma_1\xi\|_{j}+\|\Gamma_2\zeta\|_{j})\\
&\leq\frac{C}{\lambda}(\|\Gamma_1\|_{j}\|\xi\|_{0}+\|\Gamma_1\|_0\|\xi\|_{j}+
\|\Gamma_2\|_{j}\|\zeta\|_0+\|\Gamma_2\|_0\|\zeta\|_{j})
\end{split}
\end{equation}
for $j=0, 1, 2$. Therefore we need to estimate $\|\Gamma_i\|_{j}$ for  $i=1,2$ and $j=0, 1, 2$, where we refer, with a slight abuse of notation, to the $C^j$-norms in $x\in\Omega$ of the composition $x\mapsto \Gamma_i(\tilde{a}(x), \lambda x\cdot\nu)$. 
Using Lemma \ref{l:Gamma} we deduce for $i=1,2$
\begin{equation}\label{e:Gammaesti}
\begin{split}
\|\Gamma_i\|_0+\|\partial_t\Gamma_i\|_0&+\|\partial_t^2\Gamma_i\|_0\leq C\|\tilde{a}\|_0\leq C(\gamma)\eps^{1/2},\\
\|\Gamma_i\|_1&\leq \|\partial_t\Gamma_i\|_0\lambda+\|\partial_s\Gamma_i\|_0\|\nabla\tilde{a}\|_0\\
&\leq C(\gamma)\eps^{1/2}\lambda+C(M, \gamma)\eps^{1/2}\theta\\
&\leq C(\gamma)\eps^{1/2}\lambda,\\
\|\partial_t\Gamma_i\|_1&\leq \|\partial_t^2\Gamma_i\|_0\lambda+\|\partial_s\partial_t\Gamma_i\|_0\|\nabla\tilde{a}\|_0\\
&\leq C(\gamma)\eps^{1/2}\lambda,
\end{split}
\end{equation}
where we have used that $\lambda\geq C(M,\gamma)\theta$ - this can be ensured by an appropriate choice of $c_0$. 
Similarly, we also have
\begin{equation}\label{e:Gammaest12}
\begin{split}
\|\partial_s\Gamma_1\|_0&\leq C\|\tilde{a}\|_0\leq C(\gamma)\eps^{1/2},\\
\|\partial_s\Gamma_2\|_0&\leq C,\\
\|\partial_s\Gamma_1\|_1&\leq \|\partial_t\partial_s\Gamma_1\|_0\lambda+\|\partial_s^2\Gamma_1\|_0\|\nabla\tilde{a}\|_0\leq C(\gamma)\eps^{1/2}\lambda,\\
\|\partial_s\Gamma_2\|_1&\leq \|\partial_t\partial_s\Gamma_2\|_0\lambda+\|\partial_s^2\Gamma_2\|_0\|\nabla\tilde{a}\|_0\leq C(\gamma)\lambda.
\end{split}
\end{equation}
Consequently, we derive
\begin{align*}
\|v-u\|_0&\leq C(\gamma)\eps^{1/2}\lambda^{-1},\\
\|v-u\|_1&\leq C(\gamma) \eps^{1/2}+C(M,\gamma)\eps^{1/2}\delta^{1/2}\theta\lambda^{-1}\\
&\leq C(\gamma)\eps^{1/2},\\
\|v-u\|_2&\leq C(\gamma)\eps^{1/2}\lambda+C(M,\gamma)\eps^{1/2}\delta^{1/2}\theta\\
&\leq C(\gamma)\eps^{1/2}\lambda.
\end{align*}
Summarizing, we arrive at \eqref{e:step-2}, and since $\eps^{1/2}\lambda\geq M\delta^{1/2}\theta$, also at \eqref{e:step-3}.

\smallskip

Next, we derive estimates on the metric error. We calculate:
\begin{align*}
\nabla v=&\nabla u+(\partial_t\Gamma_1\xi\otimes\nu+\partial_t\Gamma_2\zeta\otimes\nu)+
\frac{1}{\lambda}(\Gamma_1\nabla\xi+\Gamma_2\nabla\zeta)\\
&+\frac{1}{\lambda}(\partial_s\Gamma_1\xi\otimes\nabla\tilde{a}+\partial_s\Gamma_2\zeta\otimes\nabla\tilde{a})\\
=&\nabla u+A+E_1+E_2,
\end{align*}
where we have set
$$
A=\partial_t\Gamma_1\xi\otimes\nu+\partial_t\Gamma_2\zeta\otimes\nu,\quad  E_1=\frac{1}{\lambda}(\Gamma_1\nabla\xi+\Gamma_2\nabla\zeta)
$$
and $E_2=E_2^{(1)}+E_2^{(2)}$ with
$$
E_2^{(1)}=\frac{1}{\lambda}\partial_s\Gamma_1\xi\otimes\nabla\tilde{a},\quad
E_2^{(2)}=\frac{1}{\lambda}\partial_s\Gamma_2\zeta\otimes\nabla\tilde{a}.
$$
Using \eqref{e:xizeta} and Lemma \ref{l:Gamma}, we have
$$
\nabla \tilde{u}^TA+A^T\nabla\tilde{u}+A^TA=a^2\nu\otimes\nu
$$
and
$$
\nabla\tilde u^TE_2^{(2)}=0.
$$
Therefore we may write, using the notation $\textrm{sym}(B)=(B+B^T)/2$,
\begin{equation}\label{e:metricerror1}
\begin{split}
&\nabla  v^T\nabla v-(\nabla u^T\nabla u+a^2\nu\otimes\nu)=\textrm{sym}\left[(\nabla u-\nabla\tilde{u})^T(A+E_2^{(2)})\right]\\
+&\textrm{sym}\left[A^TE_2^{(2)}\right]
+\textrm{sym}\left[(\nabla u+A)^T(E_1+E_2^{(1)})\right]+(E_1+E_2)^T(E_1+E_2).
\end{split}
\end{equation}
Using the estimates \eqref{e:xizeta-est}, \eqref{e:acm}, \eqref{e:Gammaesti} and \eqref{e:Gammaest12} we obtain
\begin{align*}
\|A\|_0&\leq C(\gamma)\eps^{1/2},\\
\|E_1\|_0&\leq C(M,\gamma)\lambda^{-1}\eps^{1/2}\delta^{1/2}\theta,\\	
\|E_2^{(1)}\|_0&\leq C(M,\gamma)\lambda^{-1}\eps\theta,\\
\|E_2^{(2)}\|_0&\leq C(M,\gamma)\lambda^{-1}\eps^{1/2}\theta.
\end{align*}
Using that $\eps\leq \delta$ and $\lambda\geq \theta$, we deduce
\begin{equation}\label{e:metricerror0}
\|\nabla v^T\nabla v-(\nabla u^T\nabla u+a^2\nu\otimes\nu)\|_0\leq C(M, \gamma)\eps^{1/2}\delta^{1/2}\lambda^{-1}\theta.
\end{equation}
Similarly, using the Leibniz-rule we obtain
\begin{align*}
\|A\|_1&\leq C(\|\partial_t\Gamma_1\|_{1}\|\xi\|_0+\|\partial_t\Gamma_2\|_{1}\|\zeta\|_0+\|\partial_t\Gamma_1\|_{0}\|\xi\|_{1}+\|\partial_t\Gamma_2\|_{0}\|\zeta\|_{1}).\\
&\leq C(\gamma, M)(\eps^{1/2}\lambda+\eps^{1/2}\delta^{1/2}\theta)\\
&\leq C(\gamma, M)\eps^{1/2}\lambda,\\
\|E_1\|_1&\leq \frac{C}{\lambda}(\|\Gamma_1\|_{1}\|\nabla\xi\|_0+\|\Gamma_1\|_{0}\|\nabla\xi\|_1+\|\Gamma_2\|_{1}\|\nabla\zeta\|_0+\|\Gamma_2\|_{0}\|\nabla\zeta\|_1)\\
&\leq C(\gamma, M)\eps^{1/2}\delta^{1/2}\theta,\\	
\|E_2^{(1)}\|_1&\leq \frac{C}{\lambda}(\|\partial_s\Gamma_1\|_{1}\|\xi\|_0\|\nabla\tilde{a}\|_0
+\|\partial_s\Gamma_1\|_0\|\xi\|_{1}\|\nabla\tilde{a}\|_0+\cdots+\|\partial_s\Gamma_1\|_0\|\zeta\|_0\|\nabla\tilde{a}\|_{1})\\
&\leq\frac{C(\gamma, M)}{\lambda}(\eps\lambda\theta+\eps\delta^{1/2}\theta\tilde\theta)\\
&\leq C(\gamma, M)\eps\theta,\\
\|E_2^{(2)}\|_1&\leq \frac{C}{\lambda}(\|\partial_s\Gamma_2\|_{1}\|\xi\|_0\|\nabla\tilde{a}\|_0
+\|\partial_s\Gamma_2\|_0\|\xi\|_{1}\|\nabla\tilde{a}\|_0+\cdots+\|\partial_s\Gamma_2\|_0\|\zeta\|_0\|\nabla\tilde{a}\|_{1})\\
&\leq\frac{C(\gamma, M)}{\lambda}(\eps^{1/2}\lambda\theta+\eps^{1/2}\delta^{1/2}\theta\tilde\theta)\\
&\leq C(\gamma, M)\eps^{1/2}\theta.
\end{align*}
Differentiating \eqref{e:metricerror1}, collecting terms and using the inequalities \eqref{e:step-parameter-inequalities}
we deduce
\begin{equation*}
\|\nabla v^T\nabla v-(\nabla u^T\nabla u+a^2\nu\otimes\nu)\|_{1}\leq C(M,\gamma)\eps^{1/2}\delta^{1/2}\theta.
\end{equation*}
This concludes the verification of \eqref{e:step-4}.

Finally, we verify that $v$ is an immersion. From \eqref{e:metricerror0} it follows that
\begin{align*}
\|\nabla v^T\nabla v-(\nabla u^T\nabla u+a^2\nu\otimes\nu)\|_0\leq \frac{1}{2\gamma},
\end{align*}
provided we choose $c_0\geq 2\gamma C(M,\gamma)$. Using \eqref{e:step-3a} and $\eps\leq 1$ we observe
$$0\leq a^2\nu\otimes\nu\leq \frac{\gamma}{2}\textrm{Id},$$ so that
from \eqref{e:step-1a} we readily deduce \eqref{e:step-0}.
This concludes the proof.

\end{proof}

Next we show how to utilize Proposition \ref{p:step} to ``add'' a term of the form 
$$
\sum_{k=1}^Na_k^2\nu_k\otimes \nu_k
$$
to the metric. This corresponds to a ``stage'' in the terminology of \cite{Nash54}, compare also with Proposition 4 in \cite{CDS12}. 

\begin{proposition}\label{p:stage}{\bf[Stage]}
Let $u\in C^{2}(\Omega, \R^{n+1})$ be an immersion, $\nu_k\in S^{n-1}$ and $a_k\in C^{2}(\Omega)$ for $k=1,\dots,N$. Assume that
\begin{align}
\frac{1}{\gamma}\textrm{Id}\leq\nabla u^T&\nabla u\leq\gamma \textrm{Id}\quad\textrm{ in }\Omega\label{e:stage-1a}\\
\quad \|u\|_{2}&\leq M\delta^{1/2}\theta,\label{e:stage-2a}\\
\|a_k\|_0\leq(\tfrac{1}{2}\gamma\eps)^{1/2},\quad \|a_k\|_1&\leq M\eps^{1/2}\theta, \quad \|a_k\|_2\leq M\eps^{1/2}\theta\tilde\theta,\label{e:stage-3a}
\end{align}
for some $M,\gamma\geq 1$, $\eps\leq\delta\leq 1$ and $\theta\leq\tilde\theta$.
Then there exists a constant $c_1=c_1(M,\gamma)$ such that for any $K\geq c_1\tilde\theta\theta^{-1}$
there exists an immersion $v\in C^2(\Omega, \R^{n+1})$ such that
\begin{align}
v&=u\textrm{ on }\Omega\setminus \bigcup_k\textrm{supp }a_k,\label{e:stage-1}\\
\|v-u\|_j&\leq \overline{M}\eps^{1/2}(\eps^{-1/2}\delta^{1/2}\theta K)^{j-1}\quad j=0,1,\label{e:stage-2}\\
\|v\|_2&\leq \overline{M}\delta^{1/2}\theta K^{N}.\label{e:stage-3}
\end{align}
Furthermore, there exists $\mathcal{E}\in C^1(\Omega,\R^{n\times n}_{sym})$ such that
$$
\nabla v^T\nabla v=\nabla u^T\nabla u+\sum_{k=1}^N a_k^2\nu_k\otimes\nu_k+\mathcal{E}\quad\textrm{ in }\Omega
$$
with
\begin{equation}\label{e:stage-4}
\begin{split}
\|\mathcal{E}\|_0&\leq C(M,\gamma)\eps K^{-1}, \\
\|\mathcal{E}\|_1&\leq C(M,\gamma)\eps^{1/2}\delta^{1/2}\theta K^{N-1}\,.
\end{split}
\end{equation}
The constant $\overline{M}$ depends only on $\gamma$ and $N$.
\end{proposition}

\begin{proof}
Set $u_0=u$, $\gamma_1=\gamma$ and construct  $u_1$  by applying Proposition \ref{p:step} to $u_0$ and $a_1$ with constants 
$$
\eps_1=\eps,\quad\delta_1=\delta, \quad \lambda_1=\theta K\frac{\delta^{1/2}}{\eps^{1/2}}, \quad \theta_1=\theta,\quad \tilde\theta_1=\tilde\theta.
$$
Then conditions \eqref{e:step-1a}-\eqref{e:step-3a} are satisfied, and we may in addition choose $c_1$ so that 
\eqref{e:constraint-on-lambda} is fulfilled for $K\geq c_1\tilde\theta\theta^{-1}$. We obtain $u_1$ with the properties 
\begin{equation}\label{e:step1st}
\begin{split}
&\|u_1-u_0\|_j\leq M_1\eps^{1/2}\lambda_1^{j-1}, j=0, 1,2, \\
&\|u_1\|_2\leq M_1\eps^{1/2}\lambda_1,\\
&\|\nabla u_1^T\nabla u_1-(\nabla u_{1}^T\nabla u_{1}+a_1^2\nu_1\otimes\nu_1)\|_j\leq C(M,\gamma_1)\eps^{1/2}\delta^{1/2}\theta_1\lambda_1^{j-1}, j=0, 1,
\end{split}
\end{equation}
with $M_1$ depending only on $\gamma$. Note that this time $u_1$ and $a_2$ satisfy once again the assumptions  \eqref{e:step-1a}-\eqref{e:step-3a} of Proposition \ref{p:step} with 
$$
\eps_2=\delta_2=\eps,\quad \gamma_2=2\gamma_1,\quad \theta_2=\tilde\theta_2=\lambda_1,
$$
and hence we may again apply Proposition \ref{p:step} with $\lambda_2=K\lambda_1$. It is easy to verify that condition \eqref{e:constraint-on-lambda} is then valid. Thus, in this way we may successively apply Proposition \ref{p:step} to obtain $u_2, \cdots, u_{N}$  with constants 
$$
\eps_k=\delta_k=\eps,\quad \lambda_k=\lambda_1 K^{k-1}, \quad \theta_{k}=\tilde\theta_k=\lambda_{k-1},\quad \gamma_k=2\gamma_{k-1}.
$$
We conclude for $k=2,\dots,N$
\begin{equation}\label{e:stepkth}
\begin{split}
&\|u_{k}-u_{k-1}\|_j\leq \overline{M}\eps^{1/2}\lambda_k^{j-1}, j=0, 1,2,\\
&\|u_{k}\|_2\leq\overline{M}\eps^{1/2}\lambda_k,\\
&\|\nabla u_{k}^T\nabla u_{k}-(\nabla u_{k-1}^T\nabla u_{k-1}+a_k^2\nu_k\otimes\nu_k)\|_j\leq C(M,\gamma_k)\eps\lambda_{k-1}\lambda_k^{j-1}, j=0, 1.
\end{split}
\end{equation}
We obtain an immersion $v=u_{N}\in C^2(\Omega,\R^{n+1})$ satisfying, for $j=0,1,2$,
\begin{equation*}
\|v-u\|_j\leq\sum_{k=1}^{N}\|u_k-u_{k-1}\|_j\leq\overline{M}\sum_{k=1}^{N}\eps^{1/2}\lambda_k^{j-1}\,.
\end{equation*}
Hence \eqref{e:stage-2} easily follows, with a possibly larger constant $\overline{M}$. Moreover, we further deduce
\begin{align*}
\|v\|_{2}&\leq\|u\|_{2}+\|v-u\|_{2}\leq M\delta^{1/2}\theta+\overline{M}\eps^{1/2}\sum_{k=1}^{N}\lambda_k\\
&\leq M\delta^{1/2}\theta+C\overline{M}\delta^{1/2}\theta K^{N},
\end{align*}
so that \eqref{e:stage-3} follows. 

Set 
\begin{align*}
\mathcal{E}&=\nabla v^T\nabla v-\left(\nabla u^T\nabla u+\sum_{k=1}^Na_k^2\nu_k\otimes\nu_k\right)\\
&=\sum_{k=1}^{N}[\nabla u_k^T\nabla u_k-(\nabla u_{k-1}^T\nabla u_{k-1}+a_k^2\nu_k\otimes\nu_k)].
\end{align*}
Combining with \eqref{e:step1st} and \eqref{e:stepkth}  we obtain
\begin{align*}
\|\mathcal{E}\|_0&\leq C(M,\gamma)\eps^{1/2}\delta^{1/2}\theta\lambda_1^{-1}+\sum_{k=2}^{N}\eps\lambda_{k-1}\lambda_k^{-1}\\
&\leq C(M, \gamma)\eps K^{-1},\\
\|\mathcal{E}\|_1&\leq C(M,\gamma)\eps^{1/2}\delta^{1/2}\theta+\sum_{k=2}^{N}\eps\lambda_{k-1}\\
&\leq C(M, \gamma)\eps^{1/2}\delta^{1/2}\theta K^{N-1},
\end{align*}
which are  the desired estimates of the metric error \eqref{e:stage-4}.

\end{proof}

\bigskip
As a corollary, we can add a term of the form
$\rho^2(\textrm{Id}+G)$ with $|G|\leq r_0$ (c.f.~Lemma \ref{l:deco}) to the metric. 
\begin{corollary}\label{c:stage}
Let $u\in C^2(\Omega, \R^{n+1})$ be an immersion such that 
\begin{align*}
\frac{1}{\gamma}\textrm{Id}\leq \nabla u^T&\nabla u\leq\gamma \textrm{Id},\textrm{ in }\Omega\\
\|u\|_{2}&\leq M\delta^{1/2}\theta,
\end{align*}
and let $\rho\in C^1(\Omega)$, $G\in C^1(\Omega,\R^{n\times n}_{sym})$ 
such that
\begin{align*}
\|\rho\|_0&\leq (\tfrac12\gamma\eps)^{1/2}, \quad \|\rho\|_1\leq M\eps^{1/2}\theta,\\ 
\|G\|_0&\leq r_0,\quad \|G\|_1\leq M\theta
\end{align*}
with $\eps\leq\delta\leq 1$ and $M,\gamma\geq 1$. Then there exists a constant $K_0=K_0(M,\gamma)$ such that for any
$K\geq K_0$ there exists an immersion $v\in C^2(\Omega, \R^{n+1})$ such that
\begin{align}
v&=u\textrm{ on }\Omega\setminus \textrm{supp }\rho,\label{e:cstage-1}\\
\|v-u\|_j&\leq \overline{M}\eps^{1/2}(\eps^{-1/2}\delta^{1/2}\theta K)^{j-1}\quad j=0,1,\label{e:cstage-2}\\
\|v\|_2&\leq \overline{M}\delta^{1/2}\theta K^{n_*}.\label{e:cstage-3}
\end{align}
Furthermore, there exists $\mathcal{E}\in C^1(\Omega,\R^{n\times n}_{sym})$ such that
$$
\nabla v^T\nabla v=\nabla u^T\nabla u+\rho^2(\textrm{Id}+G)+\mathcal{E}\quad\textrm{ in }\Omega
$$
with
\begin{equation}\label{e:cstage-4}
\begin{split}
\|\mathcal{E}\|_0&\leq C(M,\gamma)\eps K^{-1}, \\
\|\mathcal{E}\|_1&\leq C(M,\gamma)\eps^{1/2}\delta^{1/2}\theta K^{n_*-1}\,.
\end{split}
\end{equation}
The constant $\overline{M}$ depends only on $\gamma$.
\end{corollary}

\begin{proof}
Let us fix $K\geq K_0$ with $K_0=K_0(M,\gamma)$ as in Proposition \ref{p:stage}. We start by regularizing $\rho$ and $G$ at length-scale $\ell=\tilde\theta^{-1}$, with
$$
\tilde\theta=\frac{K\theta}{K_0}\geq \theta.
$$
We obtain
\begin{align*}
\|\tilde\rho\|_0\leq (\tfrac12\gamma\eps)^{1/2}, \quad \|\tilde\rho-\rho\|_0&\leq C(M)\eps^{1/2}\theta\ell,\quad \|\tilde \rho\|_j\leq C(M)\eps^{1/2}\theta\ell^{1-j},\\
\|\tilde G\|_0\leq r_0, \quad \|\tilde G-G\|_0&\leq C(M)\theta\ell,\quad \|\tilde G\|_j\leq C(M)\theta\ell^{1-j}
\end{align*}
for $j=1,2$. Let $h=\rho^2(\textrm{Id}+G)$ and $\tilde{h}=\tilde\rho^2(\textrm{Id}+\tilde G)$. Then 
\begin{equation}\label{e:metricerror-stage1}
\begin{split}	
\|\tilde h-h\|_0&\leq \|\tilde \rho^2-\rho^2\|_0+\|\tilde\rho^2\tilde G-\rho^2G\|_0\\
&\leq C(M, \gamma)\eps\theta \ell \leq K_0C(M, \gamma)\eps K^{-1},\\
\|\tilde h-h\|_1&\leq C(M, \gamma)\eps\theta\leq C(M, \gamma)\eps^{1/2}\delta^{1/2}\theta.
\end{split}
\end{equation}

From Lemma \ref{l:deco} we obtain, with $n_*=\frac{n(n+1)}{2}$,
\begin{equation*}
\textrm{Id}+\tilde G=\sum_{k=1}^{n_*}\tilde{a}_k^2\nu_k\otimes\nu_k,
\end{equation*}
where
$$
\|\tilde{a}_k\|_1\leq C(M)\theta,\quad \|\tilde{a}_k\|_2\leq C(M)\theta\tilde\theta.
$$ 
Define $a_k=\tilde \rho \tilde a_k$,
so that
\begin{equation}\label{e:decompstage}
\tilde h=\sum_{k=1}^{n_*}a_k^2\nu_k\otimes\nu_k,
\end{equation}
and $a_k$ satisfies
$$
\|{a}_k\|_1\leq C(M)\eps^{1/2}\theta,\quad \|{a}_k\|_2\leq C(M)\eps^{1/2}\theta\tilde\theta.
$$
Furthermore, by taking the trace of \eqref{e:decompstage} and using that $r_0\leq 1$ we deduce 
$$
\|a_k\|_0\leq \|\textrm{tr }\tilde h\|_0^{1/2}\leq \sqrt{2n}\|\tilde \rho\|_0\leq (n\gamma\eps)^{1/2}.
$$ 
Then, by replacing $\gamma$ with $2n\gamma$ we observe that the conditions of Proposition \ref{p:stage} are satisfied with $N=n_*$. Combining the conclusions \eqref{e:stage-1}-\eqref{e:stage-4} with \eqref{e:metricerror-stage1} we deduce the desired estimates \eqref{e:cstage-1}-\eqref{e:cstage-4}.

\end{proof}

\section{Proof of Theorem \ref{t:shortextend}}
In this section, we will prove Theorem \ref{t:shortextend} through Taylor expansion and applying one stage. As is well-known, for any point in $\Sigma$, there exists a neighbourhood in $\mathcal{M}$ admitting a geodesic coordinate system $(x_1,\dots,x_n)$ such that
 $\Sigma=\{x_n=0\}$ and the metric is of the form
$$
g=\sum_{i, j=1}^{n-1}g_{ij}dx^idx^j+(dx^n)^2.
$$
In the sequel we write $x=(x',x_n)$, so that in particular the given isometric immersion $f:\Sigma\to\R^{n+1}$ is locally a function $f=f(x')$. For simplicity and without loss of generality we may assume that, in these coordinates, the neighbourhood is given by $\{x:\,|x'|<1\textrm{ and }|x_n|<1\}$. The later construction is split into two steps.
\bigskip

\emph{Step 1. Initial short extension.} This step is same as \cite{HunWas16}. We make the \emph{ansatz} 
$$
u(x)=f(x')+f_1(x')x_n+f_2(x')x_n^2+f_3(x')x_n^3
$$
with $f_m(x'), m=1, 2, 3$ to be fixed. An easy calculation gives, for any $i\leq n-1$
\begin{align*}
\partial_iu(x)&=\partial_if(x')+\partial_if_1(x')x_n+\partial_if_2(x')x_n^2+\partial_if_3(x')x_n^3,\\
\partial_nu(x)&=f_1(x')+2f_2(x')x_n+3f_3(x')x_n^2,
\end{align*}
and hence, using that by assumption $f|\Sigma\to\R^{n+1}$ is isometric,
\begin{align*}
\partial_iu\cdot\partial_ju&=g_{ij}(x',0)+[\partial_if_1\cdot\partial_jf+\partial_jf_1\cdot\partial_if]x_n\\
&+[\partial_if_2\cdot\partial_jf+\partial_jf_2\cdot\partial_if+2\partial_if_1\cdot\partial_jf_1]x_n^2+o(x_n^2)\\
\partial_iu\cdot\partial_nu&=f_1\cdot\partial_if+(2f_2\cdot\partial_if+f_1\cdot\partial_if_1)x_n\\
&+[f_1\cdot \partial_if_2+2\partial_if_1\cdot f_2+3f_3\cdot\partial_if]x_n^2+o(x_n^2),\\
\partial_nu\cdot\partial_nu&=f_1\cdot f_1+4f_1\cdot f_2x_n+(6f_1\cdot f_3+4f_2\cdot f_2)x_n^2+o(x_n^2).
\end{align*}
Comparing the Taylor expansion of $g$
\begin{equation*}
g(x)=\left(              
  \begin{array}{cc}   
g_{ij}(x', 0)+\partial_ng_{ij}(x', 0)x_n+\frac{1}{2}\partial_n^2 g_{ij}(x', 0)x_n^2+o(x_n^2)
~&~0\\
0~&~1
  \end{array}
\right),
\end{equation*}
with the matrix $\nabla u^T\nabla u$, our aim is to choose $f_1$, $f_2$ and $f_3$ in such a way as to ensure that
$$g-\nabla u^T\nabla u=x_nP_1(x')+o(x_n) \text{ for }x_n>0, $$
where $P_1$ are uniformly positive definite tensors in $\Sigma$.
Thus we set 
$$
f_1=\mu,\quad f_2=-\mu,\quad f_3=0,
$$
where $\mu$ is the vectorfield satisfying \eqref{e:HW}, so that $\mu=\mu(x')$ satisfies
\begin{equation}\label{e:mu}
\mu\cdot\partial_if(x')=0,\quad |\mu|=1, \quad \mu\cdot\partial^2_{ij}f(x')-L_{ij}(x')>0,
\end{equation}
where the last inequality is understood in the sense of positive definite quadratic forms.
Since in our geodesic coordinate system $L_{ij}(x')=-\frac{1}{2}\partial_ng_{ij}(x', 0)$, our choice of $f_1,f_2,f_3$ implies
\begin{equation}\label{e:metricadapt1}
g-\nabla u^T\nabla u=x_n\left(           
  \begin{array}{cc}   
2f_1(x')\cdot\partial^2_{ij}f(x')+\partial_n g_{ij}(x', 0)~&~0\\
0~&~4
  \end{array}
\right)+O(x_n^2)
\end{equation}
For $d_0<1$ sufficiently small we see that \eqref{e:metricadapt1} is positive definite for all $0<x_n<d_0$. Therefore 
$$
\Omega=\{x:\,|x'|<1\textrm{ and }0<x_n<d_0\}
$$
is a one-sided neighbourhood of $\Sigma$ in which $u$ is strictly short and extends $f|_{\Sigma}$.

\bigskip

\emph{Step 2. Adapted short extension.} We shall utilize one stage of adding primitive metric errors to construct an adapted short immersion. First of all, the immersion $u:\Omega\to\R^{n+1}$ satisfies $u\in C^2(\Omega)$ with 
\begin{align*}
&\frac{1}{\gamma}\textrm{Id}\leq\nabla u^T\nabla u\leq\gamma \textrm{Id},\\
&\|u\|_{C^2(\Omega)}\leq M
\end{align*}
for some $\gamma,M>1$. Let 
$$
\rho^2=\frac{1}{n}\textrm{tr}(g-\nabla u^T\nabla u).
$$
From the construction in Step 1, in particular the expression \eqref{e:metricadapt1}, we deduce that there exists a constant $C\geq 1$ so that
for all $x\in\Omega$
\begin{equation}\label{e:adapted-rho}
\frac{1}{C}x_n^{1/2}\leq \rho(x)\leq Cx_n^{1/2},\quad |\nabla\rho(x)|\leq Cx_n^{-1/2},\quad |\nabla^2\rho(x)|\leq Cx_n^{-3/2}.
\end{equation}
Furthermore, there exists $\tau>0$ such that 
\begin{align*}
g-\nabla u^T\nabla u\geq 2\tau\rho^2\textrm{Id}.
\end{align*}
In particular, using Lemma 1 from \cite{Nash54} (see also Lemma 1.9 in \cite{SzLecturenotes}), we obtain the decomposition 
\begin{align*}
\frac{g-\nabla u^T\nabla u}{\rho^2}-\tau\textrm{Id}=\sum_{k=1}^N \bar{b}_k^2\varpi_k\otimes\varpi_k,
\end{align*}
for some $\varpi_k\in S^{n-1}$, $\bar{b}_k\in C^{\infty}(\Omega)$ and some integer $N$, with estimates of the form
\begin{equation}\label{e:bkbar}
\|\bar{b}_k\|_{C^j(\Omega)}\leq C(M)
\end{equation}
for $j=0,1,2$.
By setting $b_k=\bar{b}_k\rho$ we then deduce
$$
g-\nabla u^T\nabla u-\tau\rho^2\textrm{Id}=\sum_{k=1}^N b_k^2\varpi_k\otimes\varpi_k, 
$$
with estimates, for $j=0,1,2$ and $k=1,\dots,N$, 
\begin{equation}\label{e:bk}
|\nabla^jb_k(x)|\leq C(M)x_n^{1/2-j}\quad \textrm{ for }x\in\Omega. 
\end{equation}

Next, we define a Whitney-decomposition of the domain $\Omega$ as follows: Set $d_q=2^{-q}d_0$ for $q=1,2,\dots$ and define
$$
\Omega_q=\left\{x:\,|x'|<1\textrm{ and }x_n\in (d_{q+1},d_{q-1})\right\}.
$$
Moreover, let $\{\chi_q\}_{q}$ be a partition of unity on $\Omega$ subordinate to the decomposition $\Omega=\bigcup_{q=1}^\infty\Omega_q$. with the following standard properties:
\begin{itemize}
\item[(a)] $\textrm{supp}\chi_q\subset \Omega_q$, in particular $\textrm{supp}\chi_q\cap \textrm{supp}\chi_{q+2}=\emptyset$;
\item[(b)] $\sum_{q=0}^\infty\chi^2_q=1$ in $\Omega$;
\item[(c)] For any $q$ and $j=0,1,2$ we have $\|\chi_q\|_{C^j(\Omega_q)}\leq Cd_q^{-j}$.
\end{itemize} 
Thus
\begin{equation}\label{e:gdecompevenodd}
g-\nabla u^T\nabla u-\tau\rho^2\textrm{Id}=\sum_{k=1}^N \sum_{q\textrm{ odd}}(\chi_qb_k)^2\varpi_k\otimes\varpi_k+\sum_{k=1}^N \sum_{q\textrm{ even}}(\chi_qb_k)^2\varpi_k\otimes\varpi_k.
\end{equation}
We apply Proposition \ref{p:stage} with the first sum on the right hand side of \eqref{e:gdecompevenodd}. From property (c) and \eqref{e:bk} we deduce
$$
\|\chi_qb_k\|_{C^j(\Omega_q)}\leq C(M)d_q^{1/2-j},
$$
so that the assumptions \eqref{e:stage-1a}-\eqref{e:stage-3a} hold in each $\Omega_q$ with parameters
$$
\delta=1,\quad\varepsilon=d_q,\quad \theta=\tilde\theta=d_q^{-1}.
$$
Observe that, using property (a), we may ``add'' each metric term $\sum_{k=1}^N(\chi_qb_k)^2\varpi_k\otimes\varpi_k$ with $q$ odd in parallel. Proposition \ref{p:stage} leads, for any $K\geq K_0(M,\gamma)$, to an immersion $w\in C^2(\Omega,\R^{n+1})$ such that for all $q\in\N$
\begin{align*}
\|w-u\|_{C^0(\Omega_q)}&\leq \overline{M}d_q^2\frac{1}{K}\\
 \|w-u\|_{C^1(\Omega_q)}&\leq \overline{M}d_q^{1/2}\\
 \|w\|_{C^2(\Omega_q)}&\leq \overline{M}d_q^{-1}K^N.
 \end{align*}
Moreover
$$
\nabla w^T\nabla w=\nabla u^T\nabla u+\sum_{k=1}^N \sum_{q\textrm{ odd}}(\chi_qb_k)^2\varpi_k\otimes\varpi_k+\mathcal{E}_{\textrm{odd}}
$$
with
\begin{align*}
\|\mathcal{E}_{\textrm{odd}}\|_{C^0(\Omega_q)}&\leq C(M,\gamma)d_q\frac{1}{K}\\
\|\mathcal{E}_{\textrm{odd}}\|_{C^1(\Omega_q)}&\leq C(M,\gamma)d_q^{-1/2}K^{N-1}.
\end{align*}
Then, $w$ again satisfies the assumptions \eqref{e:stage-1a}-\eqref{e:stage-3a} in each $\Omega_q$ with parameters
$$
\delta=1,\quad\varepsilon=d_q,\quad \theta=\tilde\theta=d_q^{-1}K^N.
$$
Therefore, applying Proposition \ref{p:stage} once more (with the same $K$), with the second term in \eqref{e:gdecompevenodd} (the even $q$'s) leads to an immersion $v\in C^2(\Omega,\R^{n+1})$ with 
\begin{align*}
\|v-w\|_{C^0(\Omega_q)}&\leq \overline{M}d_q^2\frac{1}{K^{N+1}}\\
 \|v-w\|_{C^1(\Omega_q)}&\leq \overline{M}d_q^{1/2}
 \end{align*}
and 
\begin{equation}\label{e:v-c2esitmate}
 \|v\|_{C^2(\Omega_q)}\leq \overline{M}d_q^{-1}K^{2N}.
\end{equation}
Moreover
\begin{equation}\label{e:adaptedmetricerror}
\begin{split}
\nabla v^T\nabla v&=\nabla u^T\nabla u+\sum_{k=1}^N \sum_{q\in\N}(\chi_qb_k)^2\varpi_k\otimes\varpi_k+\mathcal{E}_{\textrm{odd}}+\mathcal{E}_{\textrm{even}}\\
&=\nabla u^T\nabla u+\sum_{k=1}^N b_k^2\varpi_k\otimes\varpi_k+\mathcal{E}
\end{split}
\end{equation}
with $\mathcal{E}=\mathcal{E}_{\textrm{odd}}+\mathcal{E}_{\textrm{even}}$ and
\begin{align*}
\|\mathcal{E}_{\textrm{even}}\|_{C^0(\Omega_q)}&\leq C(M,\gamma)d_q\frac{1}{K}\\
\|\mathcal{E}_{\textrm{even}}\|_{C^1(\Omega_q)}&\leq C(M,\gamma)d_q^{-1/2}K^{2N-1}.
\end{align*}
Putting things together we deduce for every $q\in\N$
\begin{align*}
\|v-u\|_{C^0(\Omega_q)}&\leq \overline{M}d_q^2\frac{1}{K}\,,\\
\|v-u\|_{C^1(\Omega_q)}&\leq \overline{M}d_q^{1/2}\,,\\
\|\mathcal{E}\|_{C^0(\Omega_q)}&\leq C(M,\gamma)d_q\frac{1}{K}\,,\\
\|\mathcal{E}\|_{C^1(\Omega_q)}&\leq C(M,\gamma)d_q^{-1/2}K^{2N-1}\,.
\end{align*}

Now we are in a position to show $v$ is our desired adapted short immersion. First of all, observe that for $x\in \Omega_q$ we have
$x_n\sim d_q\sim \rho^2(x)$. Therefore from \eqref{e:adapted-rho} and \eqref{e:v-c2esitmate}, we get \eqref{e:rho-adapt} and \eqref{e:u-adapt}. Besides,
\begin{align*}
|\mathcal{E}(x)|&\leq C(M,\gamma)K^{-1}\rho^2(x)\,,\\
|\nabla\mathcal{E}(x)|&\leq C(M,\gamma)K^{2N-1}\rho^{-1}(x).
\end{align*}
Let 
$$
G(x)=-\frac{\mathcal{E}(x)}{\tau\rho^2(x)},
$$
so that
\begin{equation*}
g-\nabla v^T\nabla v=\tau\rho^2\textrm{Id}-\mathcal{E}=\tau\rho^2(\textrm{Id}+G),
\end{equation*}
and, using \eqref{e:adapted-rho}
\begin{align*}
|G(x)|\leq C(M,\gamma)(\tau K)^{-1},\quad |\nabla G(x)|\leq C(M,\gamma)\rho^{-3}(x)\,.
\end{align*}
In particular, by choosing $K$ sufficiently large, we can ensure that \eqref{e:g-adapt} is satisfied. Finally, observe that for any $\alpha_0<1/3$ and any $x\in\Omega_q$ 
\begin{align*}
\|v-u\|_{C^{1, \alpha_0}(\Omega_q)}&\leq \|v-u\|_{C^1(\Omega_q)}^{1-\alpha_0}\|v-u\|_{C^2(\Omega_q)}^{\alpha_0}\leq C(M,\gamma,K) d_q^{(1-3\alpha_0)/2}
\end{align*}
is bounded independently of $q$. Consequently $v\in C^{1,\alpha_0}(\overline{\Omega})$. 
This concludes the proof.

\section{Proof of Theorem \ref{t:isometricextend}}\label{s:extension}
In this section, we will show that any adapted short immersion can be approximated by isometric immersions with the aid of Corollary \ref{c:stage} and some ideas in \cite{DS17}. The proof is divided into three steps.

{\bf Step 1. Parameter definition.}  Since $v$ is an adapted short immersion, we can write
$$g-\nabla v^T\nabla v=\rho_0^2(\textrm{Id}+G_0).$$
From the definition of adapted short immersion, one has
\begin{equation}\label{e:g0}
\rho_0\geq0, \quad |G_0|\leq r_1.
\end{equation}
Let
\begin{equation}\label{e:eps0a}
\eps_0=\max\{\max_{x\in\Omega}\rho_0^2(x), 1\}, \quad 0<a<\frac{1}{2},
\end{equation}
and define two sequences of constants $\{\eps_q\}, \{\theta_q\}$ as
\begin{equation}\label{e:espqthetaq}
\eps_q=\eps_0 A^{-2aq}, \quad \theta_q=A^{(n_*+a)q+3a}
\end{equation}
with large $A>1$ to be fixed during the proof. 
For future reference we define the sets
$$
\Omega_j^{(q)}=\left\{x\in\overline\Omega:\,\rho_q(x)>\frac{9}{8}\eps_{j+1}^{1/2}\right\}
$$
for any $j, q=0,1,2,\cdots$. Here $\rho_q(x), q=1, 2,\cdots$ will be defined in Step 2. Then it is easy to see that $\Omega^{(q)}_j\subset \Omega^{(q)}_{j+1}$ and 
$$\bigcup_j\Omega^{(q)}_j=\{x\in \overline\Omega:\rho_q(x)>0\}.$$ In particular, when $q=0,$ 
$\Omega^{(0)}_j=\{x\in \overline\Omega:\rho_0(x)>\frac{9}{8}\eps_{j+1}^{1/2}\}.$ Using Definition \ref{d:adapted} it is not difficult to verify that, whenever $x\in \Omega^{(0)}_j$ for some $j\geq 0$, we have 
\begin{equation}\label{e:v0rho0G0}
|\nabla^2 v(x)|\leq M\eps_j^{1/2}\theta_j, \quad |\nabla\rho_0(x)|\leq M\eps_{j+1}^{1/2}\theta_{j},\quad \quad |\nabla G_0(x)|\leq M\theta_j.
\end{equation}
In fact, if $\rho_0(x)>\frac{9}{8}\eps_{j+1}^{1/2}$ for some $j$, 
from \eqref{e:u-adapt} and using that $\eps_0\geq 1$ and $n_*\geq 1\geq 2a$, we obtain
\begin{align*}
|\nabla^2v(x)|\leq &M\rho_0(x)^{-2}\leq MA^{2aj+2a}=(MA^{-n_*j+2aj-a})\eps_0^{1/2}A^{(n_*+a)j+3a}A^{-aj}\\
\leq &M\eps_j^{1/2}\theta_j.
\end{align*}
Similarly, from \eqref{e:rho-adapt} we have
\begin{align*}
|\nabla\rho_0(x)|\leq&M\rho_0(x)^{-1}
\leq MA^{aj+a}=(MA^{aj-n_*j-a})\eps_0^{1/2}A^{-a(j+1)}A^{(n_*+a)j+3a}\\
\leq &M\eps_{j+1}^{1/2}\theta_j,
\end{align*}
and from \eqref{e:g-adapt}
\begin{align*}
|\nabla G_0(x)|\leq& M\rho_0(x)^{-3}\leq MA^{3aj+3a}=(MA^{2aj-n_*j})A^{(n_*+a)j+3a}\\
\leq &M\theta_j.
\end{align*}
Furthermore we may assume without loss of generality that $M\geq \overline{M}$, 
where $\overline{M}$ is the constant in Corollary \ref{c:stage}.


{\bf Step 2. Inductive construction.}
We now use Corollary \ref{c:stage} to construct a sequence of smooth adapted short immersions $\{v_q\}$, and corresponding $\{\rho_q\}$, $\{G_q\}$ such that in $\overline\Omega$
\begin{equation}\label{e:relation}
g-\nabla v_q^T\nabla v_q=\rho_q^2(\textrm{Id}+G_q),
\end{equation}
and the following statements hold:
\begin{itemize}
\item[$(1)_q$]  For any $x\in\overline\Omega$, 
$$
|G_q(x)|\leq r_2\quad\textrm{ and }\quad 0\leq\rho_q(x)\leq 4\eps_q^{1/2};
$$
\item[$(2)_q$]  If $\rho_q(x)\leq 2\eps_{q+1}^{1/2},$ then $|G_q(x)|\leq r_1$;
\item[$(3)_q$]  If $x\in\Omega_j^{(q)}$, for some $j\geq q$, then
\begin{align*}
|\nabla^2v_q(x)|\leq M\eps_j^{1/2}\theta_j, \quad |\nabla\rho_q(x)|\leq M\eps_{j+1}^{1/2}\theta_{j},\quad \quad |\nabla G_q(x)|\leq M\theta_j.
\end{align*}
\end{itemize}
In addition, if $x\in\overline\Omega\setminus\Omega_q^{(q)}$, $v_q(x)=v_{q-1}(x)$ and
for any $x\in\overline\Omega,$
\begin{equation}\label{e:vqdifference}
\|v_q-v_{q-1}\|_j\leq M\eps_q^{1/2}\theta_q^{j-1}, j=0, 1.
\end{equation}

\bigskip

Set $v_0=v$. From \eqref{e:g0}, \eqref{e:espqthetaq} and \eqref{e:v0rho0G0}, we see that conditions $(1)_0-(3)_0$ are satisfied by $v_0$. Next, suppose $v_q$, $\rho_q, G_q$ have been defined satisfying relation \eqref{e:relation} and conditions $(1)_q-(3)_q$ hold.   
Let $\chi(s)\in C^\infty(0, \infty)$ be a smooth cut-off function satisfying
$$
\chi(s)=\begin{cases} 1& s\geq 2,\\ 0&s\leq \frac{7}{4}.\end{cases}
$$
and set
$$
\phi_q(x)=\chi\left(\frac{\rho_q(x)}{\eps_{q+1}^{1/2}}\right), \qquad
\psi_q(x)=\chi\left(\frac{4\rho_q(x)}{3\eps_{q+1}^{1/2}}\right).
$$
Observe that $\supp\phi_q\subset\supp\psi_q\subset\Omega_q^{(q)}$, so that, using $(3)_q$
\begin{equation}\label{e:phiqpsiq}
\|\nabla \phi_q(x)\|_0\leq CM\theta_q,\quad \|\nabla \psi_q(x)\|_0\leq CM\theta_q.
\end{equation}
Hereafter $C$  denotes geometric constant.
Define
\begin{equation}\label{e:hq}
h_q:=(g-\nabla v_q^T\nabla v_q-\eps_{q+1} \textrm{Id})\phi_q^2=\tilde{\rho}_q^2(\textrm{Id}+\tilde{G}_q),
\end{equation}
with
\begin{align*}
\tilde{\rho}_q=\phi_q\sqrt{\rho_q^2-\eps_{q+1}}, \quad \tilde{G}_q=\frac{\psi_q\rho_q^2}{\rho_q^2-\eps_{q+1}}G_q.
\end{align*}
Observe that  the second equality in \eqref{e:hq} holds because $\psi_q=1$ on $\supp \phi_q$ and that, furthermore, $\tilde\rho_q$ is well-defined because $\rho_q\geq \tfrac{3}{2}\eps_{q+1}^{1/2}$ on $\supp\phi_q$. Consequently $\tilde{\rho}_q\in C^1(\Omega)$, $\tilde{G}_q\in C^1(\Omega;\R^{n\times n})$. Next we derive estimates on $\tilde{\rho}_q$ and $\tilde{G}_q$.

\emph{Estimates for $\tilde{\rho}_q$}: First of all, $0\leq \tilde{\rho}_q\leq \rho_q\leq4\eps_q^{1/2}$. Moreover, thanks to the fact that $\rho_q(x)\geq\frac{3}{2}\eps_{q+1}^{1/2}$ on $\supp \phi_q$, and $\frac{3}{2}\eps_{q+1}^{1/2}\leq\rho_q(x)\leq2\eps_{q+1}^{1/2}$ on  $\textrm{supp }\nabla\phi_q$,  using $(1)_q$ and $(3)_q$ we have
\begin{align*}
\|\nabla\tilde{\rho}_q\|_0&\leq \|\phi_q\nabla\sqrt{\rho_q^2-\eps_{q+1}}\|_0+\|\sqrt{\rho_q^2-\eps_{q+1}}\nabla\phi_q\|_0\\
&\leq CM\eps_q^{1/2}\theta_q.
\end{align*}

\emph{Estimates for $\tilde{G}_q$:} Note that 
$$
\tilde G_q=\psi_q G_q+\frac{\eps_{q+1}\psi_q}{\rho_q^2-\eps_{q+1}}G_q.
$$
Therefore, using $(1)_q$ and the fact that $\rho_q(x)\geq\frac{9}{8}\eps_{q+1}^{1/2}$ on $\supp \psi_q\subset\Omega_q^{(q)}$,  
\begin{align*}
\|\tilde{G}_q\|_0\leq\|G_q\|+\left\|\frac{\eps_{q+1}\psi_q}{\rho_q^2-\eps_{q+1}}G_q\right\|_0\leq \frac{81}{17}\|G_q\|_0\leq 5r_2\leq r_0.
\end{align*}
Moreover, 
\begin{align*}
\|\nabla\tilde{G}_q\|_0&\leq\left\|\frac{\psi_q\eps_{q+1}G_q}{(\rho_q^2-\eps_{q+1})^2}\nabla\rho_q^2\right\|_0+
\left\|\frac{\rho_q^2\psi_q}{\rho_q^2-\eps_{q+1}}\nabla G_q\right\|_0+\left\|\frac{\rho_q^2\nabla \psi_q}{\rho_q^2-\eps_{q+1}}G_q\right\|_0\\
&\leq CM\theta_q.
\end{align*}
Combining $(3)_q$ with the above estimates of $\tilde{\rho}_q, \tilde{G}_q$, we deduce that $\tilde{\rho}_q, \tilde{G}_q, v_q$ satisfy all the assumptions of Corollary \ref{c:stage} with constants 
$$
\eps=\delta=\eps_q,\quad \theta=\theta_q.
$$ Therefore, applying Corollary  \ref{c:stage} yields $v_{q+1}$ and $\mathcal{E}$, satisfying the following estimates
\begin{align}
&\|v_{q+1}-v_q\|_j\leq M\eps_q^{1/2}\theta_q^{j-1}A^{j-1}, j=0, 1;\label{e:vq+1:1}\\
&\|v_{q+1}\|_{2}\leq M\eps_q^{1/2}\theta_qA^{n_*};\label{e:vq+1:2}\\
&\|\mathcal{E}\|_0\leq C(M)\eps_qA^{-1}; \label{e:vq+1:3} \\
&\|\mathcal{E}\|_1\leq C(M)\eps_q\theta_qA^{n_*-1};\label{e:vq+1:4}
\end{align}
for some constant $C(M)$, along with the identity
\begin{align*}
\mathcal{E}=g-\nabla v_{q+1}^T\nabla v_{q+1}-[(1-\phi_q^2)(g-\nabla v_{q}^T\nabla v_q)+\phi_q^2\eps_{q+1}\textrm{Id}].
\end{align*}
Set 
\begin{align*}
\rho_{q+1}^2=\rho_q^2(1-\phi_q^2)+\eps_{q+1}\phi_q^2, \quad G_{q+1}=\frac{\rho_q^2G_q}{\rho_{q+1}^2}(1-\phi_q^2)+\frac{\mathcal{E}}{\rho_{q+1}^2},
\end{align*}
then $$g-\nabla v_{q+1}^T\nabla v_{q+1}=\rho_{q+1}^2(\textrm{Id}+G_{q+1}).$$
Thus $(\ref{e:relation})_{q+1}$ holds. We then give two remarks. The first one is that $\mathcal{E}$ and $\phi_q$ share the same support which can be seen from the construction in Corollary  \ref{c:stage}. Thus $G_{q+1}$ is well defined. The other one is  that if $ x\in\overline{\Omega}\setminus\Omega_q^{(q)},$ $\rho_q\leq\frac{9}{8}\eps_{q+1}^{1/2}$, then $\phi_q=0$ and  
$$(v_{q+1}, \rho_{q+1}, G_{q+1})=(v_q, \rho_q, G_q).$$
Using \eqref{e:espqthetaq} and  \eqref{e:vq+1:1}, we gain \eqref{e:vqdifference}. It remains to verify  $(1)_{q+1}-(3)_{q+1}.$

\emph{Verification of $(1)_{q+1}:$}  First, on $\textrm{supp }(1-\phi_q^2),$ one gets $\rho_q(x)\leq2\eps_{q+1}^{1/2}$, thus
$$\rho_{q+1}^2\geq\rho_q^2(1-\phi_q^2)+\frac{1}{4}\rho_q^2\phi_q^2\geq\frac{1}{4}\rho_q^2,$$
and $|G_q|\leq r_1$ from $(2)_q.$ So
\begin{equation}\label{e:rhoq+11}
\left|\frac{\rho_q^2G_q}{\rho_{q+1}^2}(1-\phi_q^2)\right|\leq 4r_1.
\end{equation}
Secondly, from the formula of $\rho_{q+1}$ and the fact that  $\rho_{q}(x)\geq\frac{3}{2}\eps_{q+1}^{1/2}$ when $x\in \textrm{supp }\phi_q,$ one has $\rho_{q+1}^2\geq\eps_{q+1}$ on $\textrm{supp }\phi_q.$ Besides, we gain from \eqref{e:vq+1:3}, on $\textrm{supp }\phi_q,$
\begin{equation*}
\|\mathcal{E}\|_0\leq\omega\eps_{q+1},
\end{equation*} 
for some small $\omega$ to be fixed, provided taking $A$ larger. Thus on $\textrm{supp }\phi_q$, we have
\begin{equation}\label{e:rhoq+12}
\left|\frac{\mathcal{E}}{\rho_{q+1}^2}\right|\leq\omega.
\end{equation}
Finally, from \eqref{e:rhoq+11} and \eqref{e:rhoq+12}, we gain
\begin{align*}
|G_{q+1}|\leq4r_1+\omega\leq r_2,
\end{align*}
provided $\omega<r_2-4r_1.$ 

Moreover, we can derive the estimate of $\|\rho_{q+1}\|_0$ as follows.
\begin{align*}
\rho_{q+1}^2\leq4\eps_{q+1}(1-\phi_q^2)+\eps_{q+1}\phi_q^2\leq 4\eps_{q+1},
\end{align*}
which leads to $\|\rho_{q+1}\|_0\leq 4\eps_{q+1}^{1/2},$ where we again use the fact that $\rho_q\leq2\eps_{q+1}^{1/2}$ when $x\in \textrm{supp }(1-\phi_q^2).$

\emph{Verification of $(2)_{q+1}:$} If $\rho_{q+1}(x)\leq2\eps_{q+2}^{1/2}<\frac{7}{4}\eps_{q+1}^{1/2},$ so  $\phi_q(x)=0$ at such point, and then $v_{q+1}=v_q$, which helps us get $(2)_{q+1}$ directly from $(2)_q.$

\emph{Verification of $(3)_{q+1}:$} First of all, observe that  $\rho_{q+1}(x)\geq\eps_{q+1}^{1/2}>\frac{9}{8}\eps_{q+2}^{1/2}$ when $x\in \textrm{supp }\phi_q,$ whereas we recall that
$$(v_{q+1}, \rho_{q+1}, G_{q+1})=(v_q, \rho_q, G_q) \text{ when } x\not\in  \textrm{supp }\phi_q,$$
thus we only need to show the case in which $j=q+1$ and $x\in \textrm{supp }\phi_q$. 

For $|\nabla^2 v_{q+1}|,$   using \eqref{e:eps0a}, \eqref{e:espqthetaq} and \eqref{e:vq+1:2}, one has
\begin{equation*}
\|v_{q+1}\|_{2}\leq M\eps_{q+1}^{1/2}\theta_{q+1},
\end{equation*}
after taking $A$ larger.
For $|\nabla\rho_{q+1}|,$ we can calculate 
\begin{align*}
\|\nabla\rho_{q+1}^2\|_0&\leq\|2\rho_q\nabla\rho_q(1-\phi_q^2)\|_0+\|2\phi_q(\eps_{q+1}-\rho_q^2)\nabla\phi_q\|_0\\
&\leq CM\eps_{q+1}\theta_q,
\end{align*}
which leads to 
\begin{align*}
|\nabla\rho_{q+1}|\leq\frac{|\nabla\rho_{q+1}^2|}{2\rho_{q+1}}\leq CM\eps_{q+1}^{1/2}\theta_{q}\leq M\eps^{1/2}_{q+2}\theta_{q+1},
\end{align*}
where we have used $\rho_{q+1}(x)\geq\eps^{1/2}_{q+1}$ when $x\in \textrm{supp }\phi_q,$ and taken $A$ larger.

For $|\nabla G_{q+1}|$, using \eqref{e:eps0a} and \eqref{e:vq+1:4}, one is able to obtain 
\begin{align*}
\|\nabla G_{q+1}\|_0&\leq\frac{1}{\min\rho_{q+1}^2}\left(\|G_q(1-\phi_q^2)\nabla\rho_q^2\|_0+\|\rho_q^2(1-\phi_q^2)\nabla G_q\|_0+\|\rho_q^2G_q\nabla\phi_q^2\|_0\right)\\
&\quad +\frac{\|\mathcal{E}\|_1}{\min\rho_{q+1}^2}
+\frac{\|\nabla\rho_{q+1}^2\|_0}{\min\rho_{q+1}^4}(\|\rho_q^2G_q(1-\phi_q^2)\|_0+\|\mathcal{E}\|_0)\\
&\leq CM\theta_q+\frac{C(M)\eps_q\theta_qA^{n_*-1}}{\eps_{q+1}}
\leq M\theta_{q+1},
\end{align*}
where we have utilized the fact that $\rho_{q+1}^2\geq\eps_{q+1}$ when $\mathcal{E}\neq0.$ 

{\bf Step 3. Convergence and conclusion.} Finally, we concentrate on the convergence of $\{v_q\}.$ From $(3)_q$ and $v_q=v_{q-1}$ when $x\in \overline\Omega\setminus\Omega_q^{(q)}$, we gain 
\begin{equation}\label{e:vq2}
\|v_q-v_{q-1}\|_{2}\leq\|v_q\|_2+\|v_{q-1}\|_2\leq 2M\eps_q^{1/2}\theta_q.
\end{equation}
Interpolation $\|v_{q}-v_{q-1}\|_1\leq M\eps_q$  and \eqref{e:vq2} gives
$$\|v_q-v_{q-1}\|_{1+\alpha}\leq 2M^2\eps_q^{1/2}\theta_q^\alpha\leq K(M, \eps_0)A^{((n_*+a)\alpha-a)q},$$
to make which convergent when $q$ goes to infinity, we should take
$$\alpha<\frac{a}{n_*+a}\rightarrow\frac{1}{2n_*+1},\text{ as } a\rightarrow\frac{1}{2}.$$
Then $$\sum_{q>m}\|v_q-v_{q-1}\|_{1+\alpha}\rightarrow0, \text{as } m\rightarrow\infty,$$
provided taking $\alpha\in[0,\frac{1}{2n_*+1})$. Obviously such $\alpha\leq\alpha_0,$ thus the sequence of adapted short immersions $\{v_q\}$ is convergent in $C^{1, \alpha}(\overline\Omega)$.
Let $\bar{v}$ be the limit, then
$$\|\bar{v}-v_0\|_0\leq\sum_{q=1}^{\infty}\|v_q-v_{q-1}\|_0\leq\sum_{q=1}^\infty M\eps_q^{1/2}\theta_q^{-1}
\leq 2M\eps_0^{1/2}A^{-n_*-2a}.$$
 Hence, upon taking $A$ larger, one can get
$$\|\bar{v}-v\|_0=\|\bar{v}-v_0\|_0\leq\epsilon.$$
Furthermore,
\begin{equation}
\|g-\nabla\bar{v}^T\nabla\bar{v}\|_0\leq\lim_{q\rightarrow\infty}(1+r_0)n\rho_q^2\leq\lim_{q\rightarrow\infty}32n\eps_q=0,
\end{equation}
which means that $\bar{v}$ is isometric in $\overline\Omega.$  Therefore, we get our desired isomeric immersion.

\section*{Acknowledgments}

The authors would like to thank the hospitality of the Max-Plank Institute of Mathematics in the Sciences, and gratefully acknowledge the support of the ERC Grant Agreement No. 724298.
%


\end{document}